\documentclass[11pt,oneside]{article}
\usepackage[centertags]{amsmath}
\usepackage{amssymb}
\usepackage{amsthm}
\usepackage{amsfonts}
\usepackage{graphicx}
\usepackage{hyperref}
\providecommand{\U}[1]{\protect\rule{.1in}{.1in}}
\newtheorem{theorem}{Theorem}[section]
\newtheorem{corollary}{Corollary}[section]
\newtheorem{remark}{Remark}[section]

\newtheorem{proposition}{Proposition}[section]
\newtheorem{definition}{Definition}[section]

\newcommand{\punt}{\boldsymbol{.}}

\newcommand{\pibs}{\boldsymbol{\pi}}
\newcommand{\cbs}{\boldsymbol{c}}
\newcommand{\ibs}{\boldsymbol{i}}
\newcommand{\kbs}{\boldsymbol{k}}
\newcommand{\jbs}{\boldsymbol{j}}
\newcommand{\ml}{\mathfrak{m}}
\newcommand{\tbs}{\boldsymbol{t}}
\newcommand{\mubs}{\boldsymbol{\mu}}
\newcommand{\nubs}{\boldsymbol{\nu}}
\newcommand{\lambdabs}{\boldsymbol{\lambda}}
\newcommand{\taubs}{\boldsymbol{\tau}}
\newcommand{\mmodels}{\boldsymbol{\vdash}}
\newcommand{\tworow}[2]{\genfrac{}{}{0pt}{}{#1}{#2}}
\begin{document}

\title{A new algorithm for computing the multivariate Fa\`a di Bruno's formula}
\author{E. Di Nardo \thanks{Dipartimento di Matematica e Informatica, Universit\`a
degli Studi della Basilicata, Viale dell'Ateneo Lucano 10, 85100 Potenza,
Italia, elvira.dinardo@unibas.it}, G. Guarino \thanks{Medical School, Universit\`a
Cattolica del Sacro Cuore (Rome branch), Largo Agostino Gemelli 8, I-00168, Roma,
Italy., E-mail: giuseppe.guarino@rete.basilicata.it}, D. Senato
\thanks{Dipartimento di Matematica e Informatica, Universit\`a degli Studi
della Basilicata, Viale dell'Ateneo Lucano 10, 85100 Potenza, Italia,
domenico.senato@unibas.it}}
\date{\today}
\maketitle

\begin{abstract}
A new algorithm for computing the multivariate Fa\`a di
Bruno's formula is provided. We use a symbolic approach based on the
classical umbral calculus that turns the computation of the
multivariate Fa\`a di Bruno's formula into a suitable
multinomial expansion. We propose a {\tt MAPLE} procedure whose computational times are
faster compared with the ones existing in the literature.
Some illustrative applications are also provided.

\end{abstract}

\textsf{\textbf{keywords}: multivariate composite function, Fa\`a
di Bruno's formula, multivariate cumulant, multivariate Hermite
polynomial, classical umbral calculus}
\section{Introduction}
The multivariate Fa\`a di Bruno's formula has been recently
addressed by the following two approaches. Combinatorial methods are
used by Costantine and Savits \cite{Const}, and (only in the
bivariate case) by Noschese and Ricci \cite{Ricci}. A
treatment based on Taylor series is proposed by Leipnik and
Pearce \cite{Leipnik}. We refer to this last paper for a detailed
list of references on this subject and for a detailed account of
its applications. We just mention the paper of Savits \cite{Savits}
for statistical applications.
A comprehensive survey of the use of univariate and multivariate
series approximation methods in statistics is given in \cite{Kolassa}.

Computing the multivariate Fa\`a di Bruno's formula by means of a
symbolic software can be done by recursively applying a chain
rule. Despite its conceptual plainess, applications
of the chain rule become impractical also for small values, because
the number of additive terms becomes awkward and the derivation of
the terms somewhat tiresome. Moreover the output is often untidy, so
further manipulations are required to simplify the result (see the example
in Appendix $2$). So a \lq\lq compressed\rq\rq version of the multivariate Fa\`a di Bruno's
formula becomes more attractive, as the multivariable dimensions
or the derivation order increase. Here, a \lq\lq compressed\rq\rq
 version of the multivariate Fa\`a di Bruno's formula is given by using
the umbral methods, introduced and developed in \cite{Dinardo1, Dinardo2,
SIAM}. These methods have been particularly suited in dealing
with topics where the composition of the formal power series plays
a crucial role, see for instance \cite{Bernoulli, StatComp} and
\cite{DiNardo0}. Therefore it
is quite natural to approach the multivariate Fa\`a di Bruno's
formula by means of these symbolic tools. The result is a new
algorithm based on a suitable generalization of the well-known
multinomial theorem:
$$(x_1 + x_2 + \cdots + x_n)^i = \sum_{k_1 + k_2 + \cdots + k_n=i}
{i \choose {k_1, k_2, \ldots, k_n}} x_1^{k_1}  x_2^{k_2} \cdots
x_n^{k_n}$$ where the indeterminates are replaced by symbolic
objects. Suitable choices of these objects give rise to an
efficient computation of the following compositions: univariate
with multivariate, multivariate with univariate, multivariate with
the same multivariate, multivariate with different multivariates
in an arbitrary number of components.

Finally, the connection between the  multivariate Fa\`a di Bruno's
formula and the multinomial theorem allows us to give a closed
form for the so-called {\it generalized Bell polynomial} introduced in
\cite{Const}. Umbral versions of the multivariate
Hermite polynomials are given as a special case of these generalized
Bell polynomials.

A {\tt MAPLE} implementation of all these formulae ends the paper.
Comparisons with existing algorithms, based on the chain rule, show
the improvement of the computational time due to the proposed approach.

\section{The umbral syntax}

Umbral methods consist essentially of a symbolic technique to deal
with sequences of numbers, indexed by nonnegative integers, where
the subscripts are treated as powers.

More formally an umbral calculus consists of a set
$A=\{\alpha,\beta, \ldots\},$ called \textit{alphabet}, whose
elements are named \textit{umbrae}, and a linear functional $E,$
called \textit{evaluation}, defined on a polynomial ring $R[A]$
and taking values in $R,$ where $R$ is a suitable ring. For the
purpose of this paper, $R$ denotes the real or complex field. The
linear functional $E$ is such that $E[1]=1$ and
\begin{equation}
E[\alpha^{i} \beta^{j} \cdots \gamma^{k}] = E[\alpha
^{i}] \, E[\beta^{j}] \cdots E[\gamma^{k}], \quad
{\hbox{(uncorrelation property)}} \label{(iii)}
\end{equation}
for any set of distinct umbrae in $A$ and for $i,j,k$
nonnegative integers. A sequence $a_{0}=1,a_{1},a_{2}, \ldots$ in
$R$ is umbrally represented by an umbra $\alpha$ when
$$E[\alpha^{n}]=a_{n}$$
for all nonnegative integers $n.$ The elements $\{a_{n}\}_{n \geq
1}$ are called \textit{moments} of the umbra $\alpha.$ Special
umbrae are
\begin{description}
\item{\it i)} the \textit{augmentation} umbra $\epsilon\in A,$
such that\footnote{Here $\delta_{i,j}$ denotes the Kronecker's
delta.} $E[\epsilon^{n}] = \delta_{0,n},$ for all nonnegative
integers $n;$ \item{\it ii)}  the \textit{unity} umbra $u \in A,$
such that $E[u^{n}]=1,$ for all nonnegative integers $n;$
\item{\it iii)} the \textit{singleton} umbra $\chi \in A$ such
that $E[\chi^{n}]=\delta_{1,n},$ for all integers $n \geq 1.$
\end{description}
Note that a sequence of moments can be umbrally represented by two or more
uncorrelated umbrae. Indeed, the umbrae $\alpha$ and $\gamma$ are
said to be {\it similar} when
$$E[\alpha^n]=E[\gamma^n], \qquad \hbox{(in symbols $\alpha \equiv
\gamma$)}$$ for all nonnegative integers $n$. An umbral polynomial
$p$ is such that $p \in R[A].$ The support of an umbral polynomial is
the set of all umbrae which occur. Two umbral polynomials $p$ and $q$
are said to be \textit{umbrally equivalent} if and only if
\begin{equation}
E[p]=E[q], \qquad \hbox{(in symbols $p\simeq q$).} \label{(ue)}
\end{equation}

The classical umbral calculus has reached a more
advanced level compared with the notions that we resume in the next paragraph. 
We only recall terminology, notation and basic definitions
strictly necessary to deal with the topic of this paper. We skip the proofs, the reader
interested in is referred to \cite{Dinardo1,
Dinardo2, SIAM}.
\subparagraph{Univariate umbral calculus.} The symbol $\gamma^{\punt n}$ 
denotes the product of $n$ uncorrelated umbrae similar to $\gamma,$ that is
$E[(\gamma^{\boldsymbol{.} \, n})^{k}] = g_{k}^{n}$ for all
nonnegative integers $k$ and $n,$ where $g_{k}=E[\gamma^k].$ 
The auxiliary umbra $\gamma^{\punt n}$ is called the \textit{$n$-th dot power}
of $\gamma.$ The summation of $n$ uncorrelated umbrae similar to $\gamma$ is
denoted by the auxiliary umbra $n \punt \gamma$ and called the \textit{dot
product} of the integer $n$ and the umbra $\gamma$. By using the
umbral equivalence (\ref{(ue)}), we are able to expand powers of $n \punt
\gamma.$ Indeed, let $\lambda=(1^{r_{1}},2^{r_{2}},\ldots)$ be an
integer partition of length $\nu_{\lambda}$ with multiplicities
$\ml(\lambda)=(r_1, r_2, \ldots).$ Set $\ml(\lambda)!=r_1! r_2!
\cdots$ and $\lambda!= (1!)^{r_{1}}(2!)^{r_{2}} \cdots.$ Then
powers of $n \punt \gamma$ verify the following umbral equivalence:
\begin{equation}
(n \punt \gamma)^{i} \simeq \sum_{\lambda\vdash i} \frac{i!}{\ml(\lambda)! \, \lambda!} (n)_{\nu_{\lambda}} \gamma_{\lambda}, \qquad
i=1,2,\ldots
\label{(eq:11)}
\end{equation}
where the summation is over all partitions $\lambda$ of the
integer $i,$ the symbol $(n)_{\nu_{\lambda}}$ denotes the lower
factorial and $\gamma_{\lambda} = (\gamma_{1})^{\punt
\,r_{1}}(\gamma_{2}^2)^{\punt \,r_{2}} \cdots,$ with $\gamma_1,
\gamma_2, \ldots$ uncorrelated umbrae similar to $\gamma.$ Since
$E[(n \punt \chi)^i] = (n)_i,$ equivalence (\ref{(eq:11)}) can be
rewritten as
\begin{equation}
(n \punt \gamma)^{i} \simeq \sum_{\lambda\vdash i} \frac{i!}{\ml(\lambda)! \, \lambda!}
(n \punt \chi)^{\nu_{\lambda}} \gamma_{\lambda}.
\label{(eq:12)}
\end{equation}
The main tool of the umbral syntax is summarized in the
following construction. By equivalence (\ref{(eq:11)}), $E[(n
\punt \gamma)^{i}]$ results to be a polynomial $q_i(n)=
\sum_{\lambda\vdash i} (n)_{\nu_{\lambda}} d_{\lambda}
E[\gamma_{\lambda}]$ of degree $i$ in $n,$ where $d_{\lambda}= i!
/ \ml(\lambda)! \lambda!.$ If the integer $n$ is replaced by an
umbra $\alpha,$ the umbral polynomial $q_i(\alpha)$ is
such that $q_i(\alpha) \simeq \sum_{\lambda\vdash i}
(\alpha)_{\nu_{\lambda}} d_{\lambda} \gamma_{\lambda}.$ We denote
by $\alpha \punt \gamma$  the auxiliary umbra such that $(\alpha
\punt \gamma)^i \simeq q_i(\alpha),$ for all
nonnegative integers $i.$ The umbra $\alpha \punt \gamma$ is
called the \textit{dot product} of $\alpha$ and $\gamma.$ Since
$(\alpha \punt \chi)^i \simeq (\alpha)_i,$ from equivalence
(\ref{(eq:12)}) we have
\begin{equation}
q_i(\alpha) \simeq \sum_{\lambda\vdash i} \frac{i!}{\ml(\lambda)! \, \lambda!}
\, (\alpha \punt \chi)^{\nu_{\lambda}} \, \gamma_{\lambda}.
\label{(eq:13)}
\end{equation}

We repeat this construction, replacing the umbra $\alpha$
in $q_i(\alpha)$ with the dot product of $\alpha$ and
$\delta,$ so we construct $(\alpha \punt \delta) \punt \gamma.$
Since $(\alpha \punt \delta) \punt \gamma \equiv \alpha \punt (\delta 
\punt \gamma),$ parenthesis can be avoided.

A noteworthy dot product is the auxiliary umbra $\alpha \punt
\beta \punt \gamma,$ where $\beta$ is the so-called Bell umbra. The moments
of the umbra $\beta$ are the Bell numbers. Indeed, since $\beta \punt \chi \equiv \chi
\punt \beta \equiv u$ then $\alpha \punt \beta \punt \chi \equiv
\alpha \punt u \equiv \alpha$ and from equivalence (\ref{(eq:13)})
we have
\begin{equation}
q_i(\alpha \punt \beta) \simeq (\alpha \punt \beta \punt
\gamma)^{i} \simeq \sum_{\lambda\vdash i} \frac{i! }
{\ml(\lambda)! \, \lambda!} \, \alpha^{\nu_{\lambda}} \, \gamma_{\lambda}. \label{(eq:14)}
\end{equation}
\begin{remark} \label{mf}
{\rm Equivalence (\ref{(eq:14)}) is the umbral version of the univariate Fa\`a
di Bruno's formula. In particular, let $\{a_n\},\{g_n\},\{h_n\}$ denote respectively the moments of
$\alpha, \gamma$ and $\alpha \punt \beta \punt \gamma.$ If we consider
the formal power series
$$f(\alpha,t) = 1 + \sum_{n=1}^{\infty} a_n \frac{t^n}{n!}, \,\,
f(\gamma,t) = 1 + \sum_{n=1}^{\infty} g_n \frac{t^n}{n!}, \,\,
f(\alpha \punt \beta \punt \gamma,t) = 1 + \sum_{n=1}^{\infty} h_n \frac{t^n}{n!},$$
then we have
\begin{equation}
f(\alpha \punt \beta \punt \gamma,t) = f[\alpha, f(\gamma,t)-1].
\label{faasing}
\end{equation}
So $E[(\alpha \punt \beta \punt \gamma)^n]$ is the $n$-th
coefficient of $f[\alpha, f(\gamma,t)-1].$ This is why the dot
product $\alpha \punt \beta \punt \gamma$ is called the
\textit{composition umbra} of $\alpha$ and $\gamma.$}
\end{remark}
\subparagraph{Multivariate umbral calculus.}
In the univariate classical umbral calculus, the main device is to
replace $a_n$ with $\alpha^n$ via the linear evaluation $E.$
Similarly, in the multivariate case, the main device is to replace
sequences like $\{g_{i_1, i_2, \ldots, i_n}\},$ where $(i_1, i_2,
\ldots, i_n) \in \mathbb{N}_0^n$ is a {\it multi-index}, with
a product of powers $\mu_1^{i_1} \mu_2^{i_2} \cdots \mu_n^{i_n},$
where $\{\mu_1, \mu_2, \ldots, \mu_n\}$ are umbral monomials in $R[A].$ Note
that the supports of the umbral monomials in $\{\mu_1, \mu_2, \ldots, \mu_n\}$ are not
necessarily disjoint. In order to manage a product like
$\mu_1^{i_1} \mu_2^{i_2} \cdots \mu_n^{i_n},$ as a power of 
an umbra, we will use a multi-index notation.  Let $\ibs \in
{\mathbb N}^n_0.$ We set $\ibs! = i_1! i_2! \cdots i_n!$ and
$(\mu_1, \mu_2, \ldots, \mu_n)^{\ibs} = \mu_1^{i_1} \mu_2^{i_2}
\cdots \mu_n^{i_n},$ where $\mubs=(\mu_1, \mu_2, \ldots,\mu_n)$
denotes a $n$-tuple of umbral monomials. We define the $\ibs$-th
power of $\mubs$ as follows $\mubs^{\ibs} = \mu_1^{i_1} \mu_2^{i_2}
\cdots \mu_n^{i_n}.$

A sequence $\{g_{\ibs}\}_{\ibs \in \mathbb{N}_0^n} \in R,$ with
$g_{\ibs} = g_{i_1, i_2, \ldots, i_n}$ and $g_{\bf 0} = 1,$ is
umbrally represented by the $n$-tuple $\mubs$ when
$$E[\mubs^{\ibs}] = g_{\ibs},$$
for all $\ibs \in \mathbb{N}_0^n.$ The elements
$\{g_{\ibs}\}_{\ibs \in \mathbb{N}_0^n}$ are called {\it
multivariate moments} of $\mubs.$ Two $n$-tuples $\nubs$ and
$\mubs$ of umbral monomials are said to be similar, if they
represent the same sequence of multivariate moments, in symbols
$\nubs \equiv \mubs.$ If for all $\ibs, \jbs \in \mathbb{N}_0^n,$
we have $E[\nubs^{\ibs} \mubs^{\jbs}] = E[\nubs^{\ibs}]
E[\mubs^{\jbs}],$ then $\nubs$ and $\mubs$ are said to be
uncorrelated.

In the univariate case we have defined the auxiliary umbra $n
\punt \gamma$ and the composition umbra $\alpha \punt
\beta \punt \gamma.$ We follow the same steps in the multivariate case.

Let $\{\mubs^{\prime}, \mubs^{\prime \prime}, \ldots,
\mubs^{\prime \prime \prime}\}$ be a set of $m$ uncorrelated
$n$-tuples similar to $\mubs.$ Define the {\it dot product} of $m$
and $\mubs$ as the auxiliary umbra $m \punt \mubs = \mubs^{\prime}
+ \mubs^{\prime \prime} + \cdots + \mubs^{\prime \prime \prime}$
and the $m$-th {\ \it dot power} of $\mubs$ as the auxiliary umbra
$\mubs^{\punt m} = \mubs^{\prime} \mubs^{\prime \prime} \cdots
\mubs^{\prime \prime \prime}.$

In the following we give an expansion of the $\ibs$-th power of $m
\punt \mubs$ in terms of \textit{partitions of a multi-index},
see Definition \ref{partitiondef}. To this aim, we introduce the
notion of {\it generating function} of the $n$-tuple $\mubs.$ The
exponential multivariate formal power series
\begin{equation}
e^{\mubs \, \tbs^T} = e^{\mu_1 t_1 + \mu_2 t_2 + \cdots + \mu_n
t_n} = u + \sum_{k=1}^{\infty} \, \sum_{\tworow{\ibs \in {\mathbb
N}^n_0}{|\ibs|=k}} \mubs^{\ibs}\frac{\tbs^{\ibs}}{\ibs!}
\label{(gf)}
\end{equation}
is said to be the generating function of the $n$-tuple $\mubs,$
where $\tbs = (t_1, t_2, \ldots, t_n)$ and $\tbs^T$ denotes its transpose.
Now, assume $\{g_{\ibs}\}_{\ibs \in \mathbb{N}_0^n}$ umbrally
represented by the $n$-tuple $\mubs.$ If the sequence
$\{g_{\ibs}\}_{\ibs \in \mathbb{N}_0^n}$ has exponential
multivariate generating function
$$
f(\mubs, \tbs) = 1 + \sum_{k=1}^{\infty} \,  \sum_{\tworow{\ibs
\in {\mathbb N}^n_0}{|\ibs|=k}} g_{\ibs}\frac{\tbs^{\ibs}}{\ibs!},
$$
by suitably extending the action of $E$ coefficientwise to
generating functions (\ref{(gf)}), we have $E[e^{\mubs \,
\tbs^T}]=f(\mubs, \tbs).$ Henceforth, when no confusion occurs, we
refer to $f(\mubs, \tbs)$ as the generating function of the
$n$-tuple $\mubs.$

\begin{proposition} \label{uncorrprop}
If $\{\mu_1, \ldots, \mu_n\}$ are uncorrelated umbral monomials,
then $\mubs \equiv \tilde{\mubs}_1 +  \cdots + \tilde{\mubs}_n,$ where
the vectors $\{ \tilde{\mubs}_1,  \ldots,  \tilde{\mubs}_n\}$ 
are uncorrelated and such that $\tilde{\mubs}_i=(\varepsilon, \ldots, \mu_i, 
\ldots, \varepsilon)$ is obtained from $\mubs$ by replacing
each umbral monomial with $\varepsilon,$ except the $i$-th one.
\end{proposition}
\begin{proof}
From (\ref{(gf)}), we have $f(\mubs, \tbs) = E[e^{\mubs \, \tbs^T}] = E[e^{\mu_1 t_1 + \mu_2 t_2 + \cdots + \mu_n
t_n}] = \prod_{i=1}^n E[e^{\mu_i t_i}]= \prod_{i=1}^n  f(\mu_i, t_i).$ The result follows
by observing that $f(\mu_i, t_i) = f(\tilde{\mubs}_i, \tbs)$ for $i=1,2,\cdots,n.$
\end{proof}

In order to compute the coefficients of $f(m \punt \mubs, \tbs),$
we first introduce the notion of \textit{composition of a
multi-index} and then the notion of \textit{partition of a
multi-index}.

\begin{definition} [Composition of a multi-index] A composition $\lambdabs$
of a multi-index $\ibs,$ in symbols $\lambdabs \models \ibs,$ is a
matrix $\lambdabs = (\lambda_{ij})$ of nonnegative integers and
with no zero columns such that
$\lambda_{r1}+\lambda_{r2}+\cdots+\lambda_{rk}=i_r$ for
$r=1,2,\ldots,n.$
\end{definition}
The number of columns of $\lambdabs$ is called the length of $\lambdabs$
and denoted by $l(\lambdabs).$
\begin{definition} [Partition of a multi-index] \label{partitiondef} A partition of
a multi-index $\ibs$ is a composition $\lambdabs,$ whose columns
are in lexicographic order, in symbols $\lambdabs \mmodels \ibs.$
\end{definition}
Just as it is for integer partitions, the notation $\lambdabs = (\lambdabs_{1}^{r_1}, \lambdabs_{2}^{r_2}, \ldots)$
means that in the matrix $\lambdabs$ there are $r_1$ columns equal to $\lambdabs_{1},$
$r_2$ columns equal to $\lambdabs_{2}$ and so on, with $\lambdabs_{1} <
\lambdabs_{2} < \cdots.$ The integer $r_i$ is the multiplicity of $\lambdabs_i.$ We set
$\ml(\lambdabs)=(r_1, r_2,\ldots).$
\begin{proposition} \label{aa} Let $\mubs$ be a $n$-tuple of umbral monomials. 
For $\ibs \in {\mathbb N}^n_0$ we have
\begin{equation}
(m \punt \mubs)^{\ibs} \simeq \sum_{\lambdabs \mmodels \ibs}
\frac{\ibs!}{\ml(\lambdabs)! \,  \lambdabs!} (m \punt
\chi)^{l(\lambdabs)} \mubs_{\lambdabs}, \label{(eq:15)}
\end{equation}
where the sum is over all partitions $\lambdabs$ of the multi-index $\ibs,$ and
$\mubs_{\lambdabs} = ({\mubs^{\prime}}^{\lambdabs_1})^{\punt
\,r_{1}}({\mubs^{\prime \prime}}^{\lambdabs_2})^{\punt \,r_{2}} \cdots,$
with ${\mubs^{\prime}}, {\mubs^{\prime \prime}}, \ldots$ uncorrelated $n$-tuples similar to $\mubs.$
\end{proposition}
\begin{proof}
Due to the uncorrelation property, we have $f(m \punt \mubs, \tbs)=[f(\mubs, \tbs)]^m$ and $
\{[f(\mubs, \tbs) - 1] + 1\}^m = 1 + \sum_{k=1}^m {m \choose k}[f(\mubs, \tbs)-1]^k.$ Moreover, we have
$$[f(\mubs, \tbs)-1]^k = \sum_{(\pibs_{1},\ldots, \pibs_{k})} \frac{g_{\pibs_{1}}}{\pibs_{1}!} \cdots \frac{g_{\pibs_{k}}}{\pibs_{k}!} \, \tbs^{\pibs_{1}} \cdots \tbs^{ \pibs_{k}}$$
where the sum is over all vectors $\{\pibs_{1},\ldots, \pibs_{k}\}
\in {\mathbb N}^n_0 \setminus \{\bf 0\} .$ If we denote by $\pibs$
the multi-index composition $(\pibs_{1},\ldots, \pibs_{k}) \models
\ibs,$ and we set $g_{\pibs} = g_{\pibs_{1}} \cdots g_{\pibs_{k}}$
and ${\ibs \choose \pibs} = \frac{\ibs!}{\pibs!},$ then we have
\begin{equation}
[f(\mubs, \tbs)-1]^k = \sum_{\ibs > 0} \left[ \sum_{\tworow{\pibs
\models \ibs}{l(\pibs)=k}} {\ibs \choose \pibs} g_{\pibs} \right]
\frac{{\tbs}^{\ibs}}{\ibs!}. \label{eq:16bis}
\end{equation}
Replacing  (\ref{eq:16bis}) in $f(m \punt \mubs, \tbs) =  1 + \sum_{k=1}^m \frac{(m)_k}{k!} [f(\mubs, \tbs)-1]^k,$ we have
\begin{equation}
f(m \punt \mubs, \tbs) =  1 + \sum_{\ibs > 0} \left[ \sum_{{\pibs
\models \ibs}} {\ibs \choose \pibs}
\frac{(m)_{l(\pibs)}}{l(\pibs)!} g_{\pibs} \right]
\frac{\tbs^{\ibs}}{\ibs!}. \label{eq:16}
\end{equation}
Observe that for each $\pibs \models \ibs,$ any other composition
$\taubs \models \ibs,$ obtained by permuting the columns of
$\pibs,$ is such that $g_{\pibs}=g_{\taubs}.$ If $\pibs =
(\cbs_1^{r_1}, \cbs_2^{r_2}, \ldots)$ then the number of distinct
permutations of $\pibs$ are $l(\pibs)!/\ml(\pibs)!.$ So indexing
the last summation in (\ref{eq:16}) by partitions $\lambdabs$ of
$\ibs,$ instead of compositions $\pibs$, we have
\begin{equation}
f(m \punt \mubs, \tbs) = 1 + \sum_{\ibs > 0} \left[
\sum_{\lambdabs \mmodels \ibs}
\frac{(m)_{l(\lambdabs)}}{\ml(\lambdabs)! \, \lambdabs!} \,
g_{\lambdabs} \right] \frac{\tbs^{\ibs}}{\ibs!}. \label{eq:18}
\end{equation}
Recalling that $E[(m \punt \chi)^{l(\lambdabs)}]=(m)_{l(\lambdabs)}$ and
$E[\mubs_{\lambdabs}]=g_{\lambdabs},$ equivalence (\ref{(eq:15)}) follows.
\end{proof}
\section{Multivariate Fa\`a di Bruno's formula}
We start by introducing the auxiliary umbra $\alpha \punt \beta
\punt \mubs.$ We define
\begin{equation}
(\alpha \punt \beta \punt \mubs)^{\ibs} \simeq \sum_{\lambdabs
\mmodels \ibs} \frac{\ibs!}{\ml(\lambdabs)! \, \lambdabs!} \,
\alpha^{l(\lambdabs)} \mubs_{\lambdabs}, \label{eq:17}
\end{equation}
where $\beta$ is the Bell umbra and $\alpha \in A.$
\begin{theorem}[Univariate composite Multivariate]  \label{mf1} Let
$\mubs$ be a $n$-tuple of umbral
monomials with generating function $f(\mubs,\tbs).$ The auxiliary umbra $\alpha \punt \beta \punt \mubs$
has generating function $f(\alpha \punt \beta \punt \mubs,\tbs)=f[\alpha, f(\mubs,\tbs)-1].$
\end{theorem}
\begin{proof}
The proof follows the same path outlined in the proof of Proposition \ref{aa}, by taking backward.
Indeed if we set $E[\alpha^{l(\lambdabs)}]=a_{l(\lambdabs)}$ and
$E[\mubs_{\lambdabs}]=g_{\lambdabs},$  from (\ref{eq:18}) and equivalence (\ref{eq:17}), 
the result follows taking into account (\ref{eq:16bis})
and (\ref{eq:16}) and also by observing that
\begin{eqnarray*}
  f(\alpha \punt \beta \punt \mubs, \tbs) & = & 1 +  \sum_{\ibs >
0} \left[ \sum_{\lambdabs \mmodels \ibs}
\frac{a_{l(\lambdabs)}}{\lambdabs! \ml(\lambdabs)!} g_{\lambdabs} \right] \frac{\tbs^{\ibs}}{\ibs!}
 =  1 + \sum_{k \geq 1} \frac{a_k}{k!}\sum_{\ibs > 0} \left[
\sum_{\tworow{\pibs \mmodels \ibs}{l(\pibs)=k}} {\ibs \choose
\pibs} g_{\pibs} \right] \frac{{\tbs}^{\ibs}}{\ibs!} \\
& = &  1 + \sum_{k \geq 1} \frac{a_k}{k!} [f(\mubs, \tbs)-1]^k = f[\alpha, f(\mubs,\tbs)-1].
\end{eqnarray*}
\end{proof}
Theorem \ref{mf1} and equivalence (\ref{eq:17}) 
generalize equation (\ref{faasing}) and equivalence
(\ref{(eq:14)}) respectively. Indeed the generating function
$f(\alpha \punt \beta \punt \mubs,\tbs)$ is the
composition of a univariate formal power series with a
multivariate formal power series and equivalence (\ref{eq:17})
gives its $\ibs$-th coefficient.

The following theorem characterizes the auxiliary umbra, whose
generating function is the composition of a multivariate formal
power series and a univariate formal power series.

\begin{theorem}[Multivariate composite Univariate] \label{mf2}  Let
$\mubs=(\mu_1, \ldots, \mu_n)$ be a $n$-tuple of umbral
monomials with generating function $f(\mubs,\tbs).$
The auxiliary umbra $(\mu_1 + \cdots + \mu_n) \punt \beta
\punt \gamma$ has generating function
\begin{equation}
f[(\mu_1 + \cdots + \mu_n) \punt \beta \punt \gamma, t]= f[\mubs,
(f(\gamma,t)-1, \ldots, f(\gamma,t)-1)]. \label{(mff2)}
\end{equation}
\end{theorem}
\begin{proof}
From Remark \ref{mf}, we have $f[(\mu_1 + \cdots + \mu_n) \punt
\beta \punt \gamma, t] = E\{ \exp[ (\mu_1 + \cdots + \mu_n)
(f(\gamma,t)-1)]\}= E\{ \exp (\mu_1 [f(\gamma,t)-1] + \cdots +
\mu_n [f(\gamma,t)-1])\}.$ Hence, the result follows from
(\ref{(gf)}).
\end{proof}
The $i$-th coefficient of $f[(\mu_1 + \cdots + \mu_n) \punt \beta
\punt \gamma, t]$ can be computed by evaluating equivalence
(\ref{(eq:14)}) after having replaced $\alpha$ with $(\mu_1 + \cdots +
\mu_n).$

In (\ref{(mff2)}), if we replace the umbra $\gamma$ with the $n$-tuple $\nubs=(\nu_1,
\dots, \nu_n),$ we obtain a characterization of
an auxiliary umbra whose generating function is the composition
of a multivariate formal power series  and  a different
multivariate formal power series. This characterization is given
in the following theorem, where the first multivariate formal
power series is umbrally represented by the $n$-tuple $\mubs$ and
the second multivariate formal power series is umbrally
represented by the $n$-tuple $\nubs.$

\begin{theorem}[Multivariate composite Multivariate] \label{mf3}  Let
$\mubs=(\mu_1, \ldots, \mu_n)$ be a $n$-tuple of umbral
monomials with generating function $f(\mubs,\tbs)$ and
$\nubs=(\nu_1, \ldots, \nu_n)$ be a $n$-tuple of umbral
monomials with generating function $f(\nubs,\tbs).$
The auxiliary umbra $(\mu_1 + \cdots + \mu_n) \punt \beta
\punt \nubs$ has generating function
$$f[(\mu_1 + \cdots + \mu_n) \punt \beta \punt \nubs, \tbs]=
f[\mubs, (f(\nubs,\tbs)-1, \ldots, f(\nubs,\tbs)-1)].$$
\end{theorem}
The $\ibs$-th coefficient of $f[(\mu_1 + \cdots + \mu_n) \punt \beta \punt \nubs, \tbs]$
can be computed by evaluating equivalence
(\ref{eq:17}) after having replaced $\mubs$ with $\nubs$ and
the umbra $\alpha$ with $(\mu_1 + \cdots + \mu_n).$

In Theorem \ref{mf3}, we can also choose a $m$-tuple
$\nubs=(\nu_1, \dots, \nu_m),$ with $m \ne n.$ To this aim, from
now on we denote by $\tbs_{(n)}$ the vector $(t_1, t_2, \ldots,
t_n)$ of length $n.$

\begin{theorem}[Multivariate composite Multivariate] \label{mf4}  Let
$\mubs=(\mu_1, \ldots, \mu_n)$ be a $n$-tuple of umbral
monomials with generating function $f(\mubs,\tbs_{(n)})$ and
$\nubs=(\nu_1, \ldots, \nu_m)$ be a $m$-tuple of umbral
monomials with generating function $f(\nubs,\tbs_{(m)}).$
The auxiliary umbra $(\mu_1 + \cdots + \mu_n) \punt \beta
\punt \nubs$ has generating function
\begin{equation}
f[(\mu_1 + \cdots + \mu_n) \punt \beta \punt \nubs, \tbs_{(m)}]=
f[\mubs, (f(\nubs,\tbs_{(m)})-1, \ldots, f(\nubs,\tbs_{(m)})-1)].
\label{(mff4)}
\end{equation}
\end{theorem}

The auxiliary umbra $(\mu_1 + \cdots + \mu_n) \punt \beta \punt
\nubs$ is a first umbral expression of the generalized Bell
polynomials, introduced in \cite{Const}. A more general expression
can be obtained from equation (\ref{(mff4)}), by replacing each
occurrence of the  umbral monomial $\nubs$ with a $m$-tuple of
umbral monomials. In other words, we deal with the composition of a
multivariate formal power series in $n$ variables and $n$
distinct multivariate formal power series in $m$ variables.

\begin{theorem}[Multivariate composite different multivariates] \label{mf5}   Let
$\mubs=(\mu_1, \ldots, \mu_n)$ be a $n$-tuple of umbral
monomials with generating function $f(\mubs,\tbs_{(n)})$ and
$\{\nubs_1, \nubs_2, \ldots, \nubs_n\}$ be a set of $m$-tuples of umbral
monomials with generating function $f(\nubs_i,\tbs_{(m)}),$ for
$i=1,2,\cdots,n.$ The auxiliary umbra $\mu_1 \punt \beta \punt \nubs_1 
+ \cdots + \mu_n \punt \beta \punt \nubs_n$ has generating function
$$f[\mu_1 \punt \beta \punt \nubs_1 + \cdots + \mu_n \punt \beta \punt
\nubs_n, \tbs_{(m)}] = f[\mubs, (f(\nubs_1,\tbs_{(m)})-1, \ldots, f(\nubs_n,\tbs_{(m)})-1)].$$
\end{theorem}
\begin{proof}
If we replace the real or complex field $R$ with the polynomial
ring $R[x_1, \ldots, x_n],$ then the uncorrelation property
(\ref{(iii)}) has to be rewritten as \cite{Dinardo1}
$$E[x_i^k x_j^m \cdots \alpha^{s} \beta^{t} \cdots] = x_i^k x_j^m \cdots E[\alpha
^{s}]E[\beta^{t}] \cdots,
$$
for any set of distinct umbrae in $A,$ and for nonnegative
integers $k, m, s, t.$ In $R[x_1, \ldots , x_n][A],$ the following
property holds
$$(x_1 + \cdots + x_n) \punt \beta \punt \alpha \equiv
(x_1 \punt \beta \punt \alpha + \cdots + x_n \punt \beta \punt \alpha).$$
In particular we have
$$f(x_1 \punt \beta \punt \nu_1 + \cdots + x_n \punt \beta \punt \nu_n, t)
=\exp\{x_1 [f(\nu_1,t)-1] +  \cdots + x_n [f(\nu_n,t)-1]\},$$
where $\{\nu_1, \ldots, \nu_n\}$ are umbral monomials
with supports not necessarily disjoint. If we replace $\nu_i$ with the $m$-tuple $\nubs_i,$ then we still have
\begin{equation}
f(x_1 \punt \beta \punt \nubs_1 + \cdots + x_n \punt \beta \punt \nubs_n, \tbs_{(m)})
=\exp\left\{ \sum_{i=1}^n x_i [f(\nubs_i,\tbs_{(m)}) - 1] \right\}.
\label{(belgen)}
\end{equation}
Equation (\ref{(belgen)}) still holds if we replace the variables
$\{x_1, \ldots, x_n\}$  with the symbols $\{\mu_1, \ldots, \mu_n\}$
and define the evaluation $E$ on the polynomial ring $R[\mu_1, \ldots , \mu_n][A]$
taking values in $R[\mu_1, \ldots , \mu_n],$ that is
\begin{equation}
f(\mu_1 \punt \beta \punt \nubs_1 + \cdots + \mu_n \punt \beta \punt \nubs_n, \tbs_{(m)})
=\exp\left\{ \sum_{i=1}^n \mu_i [f(\nubs_i,\tbs_{(m)}) - 1] \right\}.
\label{(belgen1)}
\end{equation}
Now, if we apply the linear functional $E,$
defined on the polynomial ring $R[\mu_1, \ldots , \mu_n]$
and taking values in $R,$ to equation (\ref{(belgen1)}) we have
$$E[f(\mu_1 \punt \beta \punt \nubs_1 + \cdots + \mu_n \punt \beta \punt \nubs_n, \tbs_{(m)})]
=f[\mubs, (f(\nubs_1,\tbs_{(m)})-1, \ldots, f(\nubs_n,\tbs_{(m)})-1)],$$ by which the result follows.
\end{proof}
In \cite{Const}, the $\ibs$-th coefficient $B_{\ibs}(x_1, \ldots, x_n)$ of $\exp\left\{
\sum_{i=1}^n x_i [f(\nubs_i,\tbs_{(m)}) - 1] \right\}$ is called
{\it generalized Bell polynomial}. From equation (\ref{(belgen)}),
the following corollary is immediately stated.
\begin{corollary}[Umbral representation of generalized Bell polynomials]  If
$B_{\ibs}(x_1, \ldots, x_n)$ denotes the generalized Bell
polynomial, then we have
$$B_{\ibs}(x_1, \ldots, x_n) \simeq (x_1 \punt \beta \punt \nubs_1 + \cdots +
x_n \punt \beta \punt \nubs_n)^{\ibs}.$$
\end{corollary}
From Theorem \ref{mf5}, the general expression of the multivariate
Fa\`a di Bruno's formula can be computed by evaluating the
generalized Bell polynomial $B_{\ibs}(\mu_1, \ldots, \mu_n),$
as the following corollary states.
\begin{corollary}[Multivariate Fa\`a di Bruno's formula] \label{MFB} If
$B_{\ibs}(x_1, \ldots, x_n)$ denotes the generalized Bell
polynomial, then we have
\begin{equation}
B_{\ibs}(\mu_1, \ldots, \mu_n) \simeq (\mu_1 \punt \beta \punt
\nubs_1 + \cdots + \mu_n \punt \beta \punt \nubs_n)^{\ibs}.
\label{(belgen2)}
\end{equation}
\end{corollary}
Note that the umbral polynomial $B_{\ibs}(\mu_1, \ldots, \mu_n)$ is umbrally equivalent to the $\ibs$-th
coefficient of the formal power series $\exp\left\{ \sum_{i=1}^n
\mu_i [f(\nubs_i,\tbs_{(m)}) - 1] \right\}$ in (\ref{(belgen1)}).
How to compute by means of a symbolic software the \lq\lq compressed\rq\rq multivariate Fa\`a di Bruno's formula
(\ref{(belgen2)}) is the object of the last section.
\section{Examples and applications.}
%
\subparagraph{Randomized compound Poisson random variable.}
As observed in \cite{Dinardo1}, the moments of a randomized Poisson random variable
are umbrally represented by the composition umbra $\alpha \punt \beta \punt \gamma.$
A randomized Poisson random variable is a random sum $S_N = X_1 +
\cdots + X_N$ of independent and identical distributed random
variables $\{X_i\},$ where $N$ is a Poisson random variable with
random parameter $Y.$ In $\alpha \punt \beta \punt \gamma$
the moments of $Y,$ if they exist, are umbrally represented by the umbra $\alpha$
 as well as the moments of $X_i,$ if they exist, are umbrally
represented by the umbra $\gamma.$
Similarly, the umbra $\alpha \punt \beta \punt \mubs$ is the
umbral counterpart of a multivariate randomized compound Poisson
random variable, that is a random sum $S_N = {\boldsymbol X}_1 + \cdots + {\boldsymbol X}_N$
of independent and identical distributed random vectors $\{{\boldsymbol X}_i\}.$
The evaluation $E,$ applied to both sides of equivalence (\ref{eq:17}),
gives the moments of a randomized compound Poisson random variable
in terms of moments of $\{{\boldsymbol X}_i\}$ and $N.$ This result is the same
stated in Theorem 4.1 of \cite{Const}, by using a different proof and different methods. Let us
underline that all auxiliary umbrae, constructed in the previous section,
admit probabilistic counterpart.
\subparagraph{Multivariate Laplace transform.} In multivariate probability theory,
the verification of sign alternation in derivatives of Laplace transform of densities
is one of the applications for which multivariate Fa\`a di Bruno's formula of all
orders is desirable. As it is well-known, the Laplace transform $L_{\boldsymbol X}(\tbs)$ 
of multivariate densities  is given by its moment generating function, with $\tbs$ replaced
by $-\tbs,$ that is $L_{\boldsymbol X}(\tbs)=f(\mubs, -\tbs),$ if the moments of the random vector
${\boldsymbol X}$ are umbrally represented by the $n$-tuple $\mubs.$ Since $f(\mubs, -\tbs) = f(-1 \punt \mubs, \tbs)
= f(-1 \punt \chi \punt \beta \punt \mubs, \tbs)$, equivalence (\ref{eq:17}) can be again
used in order to find the derivatives of Laplace transform, recalling that $E[(-1\punt \chi)^n] =
(-1)^n n!.$

\subparagraph{Multivariate cumulants.} If $\chi$ denotes the singleton umbra,
multivariate moments of $\chi \punt \mubs$ are the
so-called {\it multivariate cumulants} of $\mubs.$ Indeed from
Proposition \ref{mf1} we have
$f(\chi \punt \mubs, \tbs) = 1 + \log [f(\mubs,\tbs)],$ by
recalling that $\chi \punt \mubs \equiv \chi \punt \chi \punt \beta \punt \mubs$
and $f(\chi \punt \chi,t)=1 + \log(t+1).$
The equations giving multivariate cumulants in terms
of multivariate moments and vice-versa are considered in
\cite{McCullagh}. See \cite{Bernoulli} for a recent symbolic
treatment of this topic. Here, we remark that also equivalence
(\ref{eq:17}) allows us to express multivariate
cumulants in terms of multivariate moments, and vice-versa.
Indeed, since $\chi \punt \chi \punt \beta \punt \mubs \equiv
\chi \punt \mubs,$ multivariate cumulants in terms of multivariate 
moments are obtained by applying the
evaluation $E$ to both sides of equivalence (\ref{eq:17}), where
the umbra $\alpha$ has to be replaced by the umbra $\chi \punt \chi.$
Analogously, multivariate
moments in terms of multivariate cumulants are obtained by applying the evaluation
$E$ to both sides of equivalence (\ref{eq:17}), where the umbra
$\alpha$ has to be replaced by the umbra $u$ and the $n$-tuple $\mubs$
has to be replaced by $\chi \punt \mubs,$ taking into account that $u \punt
\beta \punt (\chi \punt \mubs) \equiv \mubs.$

The multivariate analogous of the well-known semi-invariance property of cumulants
is stated in Proposition \ref{corc}. First we have to define the \textit{disjoint sum}
$\mubs \dot{+} \nubs$ of two $n$-tuples $\mubs$ and $\nubs.$ As in the
univariate case, the  auxiliary umbra $\mubs \dot{+} \nubs$ is such that
$(\mubs \dot{+} \nubs)^{\ibs} \simeq \mubs^{\ibs} + \nubs^{\ibs}.$ Note that the
contribution of mixed products in the disjoint sum is zero.
\begin{proposition} \label{corc}
If $\mubs$ and $\nubs$ are uncorrelated $n$-tuples of umbral monomials,
then $\chi \punt (\mubs + \nubs) \equiv \chi \punt \mubs \dot{+} \chi \punt \nubs.$
\end{proposition}
\begin{proof}
The result follows by observing that
$f(\mubs \dot{+} \nubs, \tbs) = f(\mubs, \tbs) + f(\nubs, \tbs) - 1.$
\end{proof}
Thus multivariate cumulants linearize convolutions of uncorrelated $n$-tuples
of umbral monomials. By using Propositions \ref{uncorrprop} and \ref{corc}, we are able
to state that mixed cumulants are zero in the case of uncorrelated components.
\begin{corollary}
If $\{\mu_i\}_{i=1}^n$ are uncorrelated umbral monomials, then
$\chi \punt \mubs \equiv \chi \punt \tilde{\mubs}_1 \dot{+} \cdots
\dot{+} \chi \punt \tilde{\mubs}_n,$
where $\mubs = (\mu_1, \ldots, \mu_n)$ and the uncorrelated vectors $\{\tilde{\mubs}_i\}_{i=1}^n$ are
obtained from $\mubs$ by replacing each component with $\varepsilon,$ except the $i$-th one,
that is $\tilde{\mubs}_i=(\varepsilon, \ldots, \mu_i, \ldots, \varepsilon).$
\end{corollary}

\subparagraph{Multivariate Hermite polynomials.}
The $\ibs$-th Hermite polynomial $H_{\ibs}(\boldsymbol{x}, \Sigma)$ is defined as
$H_{\ibs}(\boldsymbol{x}, \Sigma) = (-1)^{|\ibs|}
D_{\boldsymbol{x}}^{(\ibs)} \phi(\boldsymbol{x}; \boldsymbol{0}, \Sigma)/{\phi(\boldsymbol{x}; \boldsymbol{0}, \Sigma)},$
where $\phi(\boldsymbol{x}; \boldsymbol{0}, \Sigma)$ denotes the multivariate gaussian
density with $\boldsymbol{0}$ mean and covariance matrix $\Sigma$ of full rank $n.$
These polynomials are orthogonal with respect to $\phi(\boldsymbol{x}; \boldsymbol{0}, \Sigma)$ if we
consider the polynomials $\tilde{H}_{\ibs}(\boldsymbol{x}, \Sigma) = H_{\ibs}(\boldsymbol{x} \Sigma^{-1}, \Sigma^{-1}),$
where $\Sigma^{-1}$ denotes the inverse matrix of $\Sigma.$ The following result is proved in \cite{Whiters}
\begin{equation}
H_{\ibs}(\boldsymbol{x}, \Sigma) = E[( \boldsymbol{x} \Sigma^{-1} + i \boldsymbol{Y})^{\ibs}] \qquad
\tilde{H}_{\ibs}(\boldsymbol{x}, \Sigma) = E[(\boldsymbol{x} + i \boldsymbol{Z})^{\ibs}]
\label{(herfin)}
\end{equation}
where $i$ is the imaginary unit and $\boldsymbol{Z} \sim N(\boldsymbol{0}, \Sigma)$ \footnote{ Here
$\boldsymbol{Z} \sim N(\boldsymbol{0}, \Sigma)$ denotes  a multivariate gaussian vector $\boldsymbol{Z}$
with $\boldsymbol{0}$ mean and covariance matrix $\Sigma.$} and  $\boldsymbol{Y} \sim N(\boldsymbol{0}, \Sigma^{-1}).$
By keeping the length of the present paper within bounds, we use (\ref{(herfin)}) in order to get the umbral
expression of multivariate Hermite polynomials.
\begin{proposition}
If $\nubs$ is a $n$-tuple of umbral monomials such that $f(\nubs, \tbs)= 1 + \frac{1}{2} \tbs \Sigma^{-1} \tbs^T$
and $\mubs$ is a $n$-tuple of umbral monomials such that $f(\mubs, \tbs)= 1 + \frac{1}{2} \tbs \Sigma \tbs^T,$
then we have
\begin{equation}
H_{\ibs}(\boldsymbol{x}, \Sigma) \simeq (- 1 \punt \beta \punt \nubs + \boldsymbol{x} \Sigma^{-1})^{\ibs} \qquad
\tilde{H}_{\ibs}(\boldsymbol{x}, \Sigma) \simeq (-1 \punt \beta \punt \mubs + \boldsymbol{x})^{\ibs}.
\label{(herfinumb)}
\end{equation}
\end{proposition}
\begin{proof}
Equivalences (\ref{(herfinumb)}) follows by observing that from (\ref{(herfin)}) we have
\begin{eqnarray}
1 + \sum_{k=1}^{\infty} \, \sum_{\tworow{\ibs \in {\mathbb
N}^n_0}{|\ibs|=k}} H_{\ibs}(\boldsymbol{x}, \Sigma) \frac{\tbs^{\ibs}}{\ibs!} & = & \exp \left( \boldsymbol{x} \Sigma^{-1} \tbs^T
- \frac{1}{2} \tbs \Sigma^{-1} \tbs^T \right), \label{(gen1)} \\
1 + \sum_{k=1}^{\infty} \, \sum_{\tworow{\ibs \in {\mathbb
N}^n_0}{|\ibs|=k}} \tilde{H}_{\ibs}(\boldsymbol{x}, \Sigma) \frac{\tbs^{\ibs}}{\ibs!} & = &
\exp \left( \boldsymbol{x}  \tbs^T
- \frac{1}{2} \tbs \Sigma \tbs^T \right). \nonumber
\end{eqnarray}
\end{proof}
In \cite{DinNierSen}, it is proved that Appell polynomials are umbrally represented by the polynomial umbra 
$x + \alpha.$ It is well-known that univariate Hermite polynomials are Appell polynomials. Then equivalences
(\ref{(herfinumb)}) show that also multivariate Hermite polynomials are of Appell type.
Let us underline that the umbra $-1 \punt \beta$ allows us a simple expression of multivariate Hermite polynomials,
without the employment of the imaginary unit of equations (\ref{(herfin)}). Moreover, since the moments of
$\boldsymbol{Z} \sim N(\boldsymbol{0},\Sigma)$  are umbrally represented by the umbra
$\beta \punt \mubs,$ then the $n$-th tuple $\mubs \equiv \chi \punt (\beta \punt \mubs)$ umbrally represents 
the multivariate cumulants of $\boldsymbol{Z}.$ 

The following proposition states that the multivariate Hermite polynomials are special generalized Bell polynomials.
\begin{proposition} \label{cor1}
If $\mubs$ is a $n$-tuple of umbral polynomials such that $f(\mubs, \tbs)= 1 + \frac{1}{2} \tbs \Sigma^{-1} \tbs$
and $\mubs_{\boldsymbol{x}}$ is a $n$-tuple of umbral polynomials such that $f(\mubs_{\boldsymbol{x}}, \tbs)=
1+ f(\mubs, \tbs + \boldsymbol{x}) - f(\mubs, \tbs),$ then $H_{\ibs}(\boldsymbol{x}, \Sigma) \simeq (-1)^{|\ibs|}
(-1 \punt \beta \punt \mubs_{\boldsymbol{x}})^{\ibs}.$
\end{proposition}
\begin{proof}
We have
$$1 + \sum_{k \geq 1} \sum_{\ibs : |\ibs|=k} (-1)^{|\ibs|} (- 1 \punt \beta \punt  \boldsymbol{\mu}_{\boldsymbol{x}})^{\ibs} \frac{\tbs^{\ibs}}{\ibs!}
= 1 + \sum_{k \geq 1} \sum_{\ibs : |\ibs|=k} (- 1 \punt \beta \punt
\boldsymbol{\mu}_{\boldsymbol{x}})^{\ibs} \frac{-\tbs^{\ibs}}{\ibs!}$$
where $-\tbs = (-t_1, -t_2, \ldots, -t_n).$ The result follows from (\ref{(gen1)}) since we have
$$1 + \sum_{k \geq 1} \sum_{\ibs : |\ibs|=k} E[(- 1 \punt \beta \punt
\boldsymbol{\mu}_{\boldsymbol{x}})^{\ibs}] \frac{-\tbs^{\ibs}}{\ibs!} = f(- 1 \punt \beta \punt
\boldsymbol{\mu}_{\boldsymbol{x}}, - \tbs) = \exp \left( \boldsymbol{x} \Sigma^{-1} \tbs^T
- \frac{1}{2} \tbs \Sigma^{-1} \tbs^T \right).$$
\end{proof}
Finally we remark that computing efficiently the multivariate Hermite polynomials via equivalences
(\ref{(herfinumb)}) or Proposition \ref{cor1} helps in constructing multivariate Edgeworth
approximation of multivariate density functions, see \cite{Kolassa}.

\section{The UMFB algorithm.}
In this last section, we present a {\tt MAPLE} algorithm for the
computation of the multivariate Fa\`a di Bruno's formula by using the umbral
equivalence (\ref{(belgen2)}). The main steps can be summarized as followed:
\begin{description}
\item{\it i)} to the right-hand-side of equivalence (\ref{(belgen2)}),
we apply the multivariate version of the well-known
multinomial theorem:
\begin{eqnarray}
&  &  \!\!\!\! \!\!\!\! \!\!\!\! \!\!\!\! \!\!\!\! \!\!\!\!
(\mu_1 \punt \beta \punt \nubs_1 + \cdots + \mu_n \punt \beta \punt \nubs_n)^{\ibs} \nonumber \\
& = &  \sum_{(\kbs_1, \ldots,\kbs_n) : \sum_{i=1}^n \kbs_i = \ibs}
{\ibs \choose {\kbs_1, \ldots, \kbs_n}} (\mu_1 \punt \beta \punt
\nubs_1)^{\kbs_1} \cdots (\mu_n \punt \beta \punt
\nubs_n)^{\kbs_n}. \label{(belgen3)}
\end{eqnarray}
The procedure {\tt mkT} finds all the vectors
$(\kbs_1, \ldots,\kbs_n)$ such that $\sum_{i=1}^n \kbs_i = \ibs.$
\item{\it ii)} Then, it is necessary to expand powers like
$(\mu_j \punt \beta \punt \nubs_j)^{\kbs_j}.$ The procedure {\tt MFB}
realizes this computation by using equivalence (\ref{eq:17}), with $\alpha$ replaced by $\mu_j.$
This procedure makes use of the procedure {\tt makeTab}, available online at
\begin{verbatim}
http://www.maplesoft.com/applications/view.aspx?SID=33039
\end{verbatim}
We will add more details on the output of procedure {\tt makeTab} later on.
\item{\it iii)} The procedure {\tt joint} provides a way to multiply the
factors $(\mu_1 \punt \beta \punt \nubs_1)^{\kbs_1} \cdots (\mu_n \punt
\beta \punt \nubs_n)^{\kbs_n}$ in (\ref{(belgen3)}), previously expanded
in {\it ii)}.
\item{\it iv)} Finally, occurrences of
products like $\mu_1^{j_1} \mu_2^{j_2} \cdots \mu_n^{j_n}$ are replaced by
$g_{j_1, j_2, \ldots, j_n}$ while occurrences of products like $(\nubs_i)_{\lambdabs},$
where $\lambdabs = (\lambdabs_1^{r_1}, \lambdabs_2^{r_2}, \ldots),$
are replaced by $(h^{(i)}_{\lambdabs_1})^{r_1} (h^{(i)}_{\lambdabs_2})^{r_2} \cdots,$ where
$h^{(i)}_{\jbs}$ denotes the $\jbs$-th coefficient of $f(\nubs_i, \tbs_{(m)}).$
\end{description}

All these steps are combined in the procedure {\tt UMFB}. Some computational
results are given in Table $1$. All tasks have been performed on a PC Pentium(R)4
Intel(R), CPU 2.3 Ghz, 2.0 GB Ram with {\tt MAPLE} version 12.0. The
procedure {\tt diff} of Table $1$ is a {\tt MAPLE} routine by which
the multivariate Fa\`a di Bruno's formula is computable making use of the
chain rule.

Now let us give more details on the quoted procedure {\tt makeTab}.
This procedure gives multiset subdivisions. A multiset $M$ is a pair $(\bar{M},f),$
where $\bar{M}$ is a set, called {\it support} of the multiset, and $f$ is a
function $f: \bar{M} \rightarrow {\mathbb N}_0.$ For each $\mu \in \bar{M},$
the integer $f(\mu)$ is called the {\it multiplicity} of $\mu.$ The notion of
multiset subdivision is quite natural and it is equivalent to split the
multiset into disjoint blocks (submultisets) whose union gives the whole multiset.
For a formal definition, the reader is referred to \cite{Bernoulli}. As example,
for the multiset $M=\{\mu_1, \mu_1, \mu_2\},$ the subdivisions are
\begin{equation}
\{\{\mu_1, \mu_1, \mu_2\}\}; \{\{\mu_1, \mu_1\},\{\mu_2\}\}; \{\{\mu_1, \mu_2,\} \{\mu_1\}\};
\{\{\mu_1\}, \{\mu_1\}, \{\mu_2\}\}.
\label{(belgen4)}
\end{equation}
If the input parameter is the vector $\ibs$, the output of the procedure {\tt makeTab}
gives all the subdivisions of a multiset having the vector $\ibs$ as vector of
multiplicities, that is $f(\mu_j)=i_j$ for $j=1,2, \ldots, n.$
At a first glance, a multiset subdivision does not seem to have any relation
with multi-index partitions. But any subdivision of a multiset $M,$ having $\ibs$ as vector of
multiplicities, corresponds to a suitable partition of the multi-index $\ibs$. As example, the multiset
$M=\{\mu_1, \mu_1, \mu_2\}$ corresponds to the multi-index $(2,1),$ and the subdivisions
in (\ref{(belgen4)}) correspond to the multi-index partitions
$${2 \choose 1}, { {2,0} \choose {0,1}}, {{1,1} \choose {1,0}}, {{1,1,0} \choose {0,0,1}}.$$
Therefore the procedure {\tt makeTab} allows us to compute all the partitions of a multi-index
$\ibs.$ On the other hand, the connection between the combinatorics of multisets and the Fa\`a di Bruno's formula 
was already remarked in \cite{Hardy}. 
\begin{table}
\begin{tabular}{ccccc} \hline
$\ibs$ & $m$ & $\#$ terms in output & Time ({\tt UMFB}) & Time ({\tt diff}) \\ \hline
$(6,5)$   & $2$ &   $14089$   &  $0.7$  & $1.6$ \\
$(7,6)$   & $2$ &   $60190$   &  $3.2$  & $29.8$ \\
$(7,7)$   & $2$ &   $123134$  &  $8.1$  & $75.4$ \\
$(5,4)$   & $3$ &   $20208$   &  $0.7$  & $2.3$ \\
$(6,5)$   & $3$ &   $122034$  &  $6.3$  & $62.5$ \\
$(5,4)$   & $4$ &   $86768$   &  $4.5$  & $26.9$ \\
$(5,4)$   & $5$ &   $288370$  &  $25.9$ & $130.9$ \\
$(4,4,3)$ & $2$ &   $95138$   &  $6.3$  & $12.3$ \\
$(4,4,4)$ & $2$ &   $257854$  &  $22.8$ & $41.1$ \\
$(4,3,3)$ & $3$ &   $313866$  &  $22.5$ & $54.5$ \\
$(4,2,2)$ & $4$ &   $106912$  &  $6.5$  & $17.3$ \\ \hline
\end{tabular}
\caption{Comparison of computational times in seconds.}
\end{table}

\section{Appendix 1.}
The {\tt UMFB} algorithm.
\begin{verbatim}
>
MFB := proc()
local n,vIndets,E;
option remember;
n:=add(args[i],i=1..nargs);
if n=0 then return(1);fi;
vIndets:=[seq( alpha[i],i=1..nargs)];
E:=add(f[nops(y[1])]*
       y[2]*
       mul(g[seq(degree(x,vIndets[i]),
                 i=1..nops(vIndets))],x=y[1]),
   y=makeTab(args));
end:
>
joint := proc()
local p1,p2,M,V;
V  := `if`(nargs=1, [[args]], [args]);
p1 := mul(add(y,y=x)!, x=ListTools[Transpose](V));
p2 := mul(x!, x=ListTools['Flatten'](V));
M  := max(seq( add(y,y=x), x=V ));
expand(p1/p2*mul(  eval(MFB(op(V[i])) ,[g=g||i,
    seq(f[j]=f||i^j,j=1..M)]), i=1..nargs ));
end:
>
mkT := proc(V,n)
local vE,L,nV;
nV:=nops(V);
vE:=[seq(alpha[i],i=1..nV)];
L:=seq( `if`(nops(x[1])<=n,x[1],NULL), x=makeTab( op(V) ));
L:=seq([seq([seq(degree(y,z),z=vE)],y=x),[0$nV]$(n-nops(x))],x=[L]);
L:=seq( op(combinat[permute](x)),x=[L]);
end:
>
UMFB := proc(V, n)
  local S,vE;
  if n=1 then
     return(expand(eval(MFB(op(V)),[g=g1]))); fi;
  vE:=[seq(f||i=1,i=1..n)];
  S:=add(joint( op(x) ), x=[mkT(V,n)]);
  add( f[ seq(degree(x,f||i),i=1..n) ]*
          eval(x,vE), x=S );
end:
\end{verbatim}
\section{Appendix 2.}
The output of the routine {\tt diff} of {\tt MAPLE}, for
$\frac{\partial^2}{\partial x1 \partial x2} f(g1(x1,x2),g2(x1,x2))$
is
\begin{eqnarray*}
&  & \textstyle{ \left( {\frac {\partial }{\partial x1}} g1 \left( x1
, x2 \right)  \right)  \left( D_{{1,1}} \right)  \left( f
 \right)  \left( g1 \left( x1, x2 \right), g2
 \left( x1, x2 \right)  \right) {\frac {\partial }{
\partial x2}}g1 \left( x1, x2 \right)} \\
& + & \textstyle{\left(
{\frac {\partial }{\partial x1}}g1 \left( x1, x2 \right)  \right)  \left( D_{{1,2}} \right)  \left( f \right)
 \left( g1 \left( x1, x2 \right) , g2 \left( {
\it x1}, x2 \right)  \right) {\frac {\partial }{\partial x2
}} g2 \left( x1, x2 \right)}  \\
& + & \textstyle{ D_{{1}} \left( f \right)
 \left( g1 \left( x1, x2 \right) , g2 \left( {
\it x1}, x2 \right)  \right) {\frac {\partial ^{2}}{\partial {
\it x1}\partial x2}}g1 \left( x1, x2 \right)} \\
& + &  \textstyle{ \left( {\frac {\partial }{\partial x1}} g2 \left( x1
, x2 \right)  \right)  \left( D_{{1,2}} \right)  \left( f
 \right)  \left( g1 \left( x1, x2 \right) , g2
 \left( x1, x2 \right)  \right) {\frac {\partial }{
\partial x2}}g1 \left( x1, x2 \right)} \\
& + & \textstyle{ \left(
{\frac {\partial }{\partial x1}} g2 \left( x1, x2 \right)  \right)  \left( D_{{2,2}} \right)  \left( f \right)
 \left( g1 \left( x1, x2 \right) , g2 \left( {
\it x1}, x2 \right)  \right) {\frac {\partial }{\partial x2
}} g2 \left( x1, x2 \right)} \\
& + & \textstyle{D_{{2}} \left( f \right)
 \left( g1 \left( x1, x2 \right) , g2 \left( {
\it x1}, x2 \right)  \right) {\frac {\partial ^{2}}{\partial {
\it x1}\partial x2}} g2 \left( x1, x2 \right)}.
\end{eqnarray*}
The output of the routine {\tt UMFB} for
$\frac{\partial^2}{\partial x1 \partial x2} f(g1(x1,x2),g2(x1,x2))$ is:
\begin{eqnarray*}
&  &
f_{{1,0}} {\it g1}_{{1,1}} + f_{{2,0}} {\it g1}_{{1,0}} {\it g1}_{{0,1}} +
f_{{0,1}} {\it g2}_{{1,1}} + f_{{0,2}} {\it g2}_{{1,0}} {\it g2}_{{0,1}} +
f_{{1,1}} {\it g1}_{{1,0}} {\it g2}_{{0,1}} + f_{{1,1}} {\it g1}_{{0,1}} {\it g2}_{{1,0}}.
\end{eqnarray*}

\end{document}